\documentclass[a4paper,11pt]{amsart}
\usepackage[matrix,arrow,tips,curve]{xy}
\usepackage[english]{babel}
\usepackage[utf8]{inputenc}
\usepackage{amsmath}
\usepackage{amssymb}
\usepackage{tikz}
\usepackage{mathrsfs}
\usepackage{enumerate}
\usepackage{graphicx}
\usepackage{hyperref}
\usepackage{verbatim}
\usetikzlibrary{trees}
\usetikzlibrary{arrows}
\usepackage{color}
\input{xy}
\xyoption{all}
\newtheorem{thm}{Theorem}[section]
\newtheorem{Lemma}[thm]{Lemma}
\newtheorem{Proposition}[thm]{Proposition}

\newtheorem*{thm*}{Theorem}

\theoremstyle{definition}

\newtheorem{Definition}[thm]{Definition}
\newtheorem{Remark}[thm]{Remark}
\newtheorem{Question}[thm]{Question}
\newtheorem{Example}[thm]{Example}

\newcommand{\Z}{\mathbb{Z}}
 
\renewcommand{\P}{\mathbb{P}}

\DeclareMathOperator{\Pic}{Pic}
\DeclareMathOperator{\VSP}{VSP}

\DeclareMathOperator{\Hilb}{Hilb}

\DeclareMathOperator{\diag}{diag}

\DeclareMathOperator{\mult}{mult}

\DeclareMathOperator{\im}{Im}

\hypersetup{pdfpagemode=UseNone}
\hypersetup{pdfstartview=FitH}
		
\begin{document}

\title[\resizebox{5.5in}{!}{Varieties of sums of powers and moduli spaces of (1,7)-polarized abelian surfaces}]{Varieties of sums of powers and moduli spaces of (1,7)-polarized abelian surfaces}

\author[Michele Bolognesi]{Michele Bolognesi}
\address{\sc Michele Bolognesi\\
IMAG - Universit\'e de Montpellier\\
Place Eug\`ene Bataillon\\
34095 Montpellier Cedex 5\\ France}
\email{michele.bolognesi@umontpellier.fr}

\author[Alex Massarenti]{Alex Massarenti}
\address{\sc Alex Massarenti\\
Universidade Federal Fluminense\\
Rua M\'ario Santos Braga\\
24020-140, Niter\'oi,  Rio de Janeiro\\ Brazil}
\email{alexmassarenti@id.uff.br}

\date{\today}

\begin{abstract}
We study the geometry of some varieties of sums of powers related to the Klein quartic. This allows us to describe the birational geometry of certain moduli spaces of abelian surfaces. In particular we show that the moduli space $\mathcal{A}_2(1,7)^{-}_{sym}$ of $(1,7)$-polarized abelian surfaces with a symmetric theta structure and an odd theta characteristic is unirational by showing that it admits a dominant morphism from a unirational conic bundle.
\end{abstract}
\maketitle

\tableofcontents

\section{Introduction}\label{intro}
Varieties of sums of powers, $\VSP$ for short, parametrize decompositions of a general homogeneous polynomial $F\in k[x_{0},...,x_{n}]_d$ as sums of powers of linear forms. They have been widely studied from both the biregular \cite{IR01}, \cite{Mu92}, \cite{RS00} and the birational viewpoint \cite{MM13}, \cite{Ma16}, and furthermore in relation to secant varieties \cite{COV17b}, \cite{COV17a}, \cite{Me06}, \cite{Me09}. 

The relation of $\mathcal{A}_2(1,7)$, the moduli space of abelian surfaces with a polarization of type $(1,7)$ and a $(1,7)$-level structure, with the variety of sums of powers of the Klein quartic dates back to the work of S. Mukai \cite{Mu92}.

In this paper we investigate the birational geometry of some moduli spaces of abelian surfaces related to $\mathcal{A}_2(1,7)$. In particular, if we endow the abelian surfaces in $\mathcal{A}_2(1,7)$ with a symmetric theta structure and a theta characteristic (odd or even), we obtain two new moduli spaces, $\mathcal{A}_2(1,7)^{-}_{sym}$ and $\mathcal{A}_2(1,7)^{+}_{sym}$, that are finite covers of degree $6$ and $10$ respectively of $\mathcal{A}_2(1,7)$. For a general introduction to these spaces see \cite[Sections 6.1.1]{BM16}. We introduce also the moduli space $\mathcal{A}_2(1,7;2,2)$ parametrizing abelian surfaces with a polarization of type $(1,7)$, a $(1,7)$-level structure and a $(2,2)$-level structure.

Furthermore, we introduce new types of varieties of sums of powers and showcase rational maps between them and our moduli spaces of abelian surfaces.
The first variety of sums of powers we take into account is a variety where we also allow ourselves to fix an order on the linear forms that make up the decomposition of the polynomial. Let $\nu_{d}^{n}:\mathbb{P}^{n}\rightarrow\mathbb{P}^{N(n,d)}$, with $N(n,d) =\binom{n+d}{d}-1$ be the degree $d$ Veronese embedding of $\mathbb{P}^n$, and let $V_{d}^{n} = \nu_{d}^{n}(\mathbb{P}^{n})$ be the corresponding Veronese variety.

\begin{Definition}
Let $F$ be a general polynomial of degree $d$ in $n+1$ variables. We define 
$$\VSP_{ord}(F,h)^{o} := \{(L_{1},...,L_{h})\in (\mathbb{P}^{n*})^{h} \: | \: 
F\in \langle L_{1}^d,...,L_{h}^d\rangle\subseteq\mathbb{P}^{N(n,d)}\}\subseteq (\mathbb{P}^{n*})^{h}$$
and $\VSP_{ord}(F,h) := \overline{\VSP_{ord}(F,h)^{o}}$ by taking the closure of $\VSP_{ord}(F,h)^{o}$ in $(\mathbb{P}^{n*})^{h}$. 
\end{Definition}

Then we consider the natural action of the symmetric group $S_h$ on $\VSP_{ord}(F,h)$ and two related variations on the classical definition of $\VSP$.

\begin{Definition}
Consider the rational action of $S_{h-1}$ on $\VSP_{ord}(F,h)$ given by permuting the linear forms $(L_2,\dots,L_h)$. The variety $\VSP_{h}(F,h)$ is the quotient 
$$\VSP_{h}(F,h) = \VSP_{ord}(F,h)/S_{h-1}.$$
If $h = 2r$ is even, we consider the rational action of $S_r\times S_r$ on $\VSP_{ord}(F,h)$ such that the first and the second copy of $S_r$ act on the first and the last $r$ linear forms respectively. Let $X_{h}(F) = \VSP_{ord}(F,h)/S_{r}\times S_r$. This space comes with a natural $S_2$-action switching $\{L_1,\dots, L_r\}$ and $\{L_{r+1},\dots,L_{h}\}$. We define 
$$\VSP^{h}(F,h):= X_h(F)/S_{2}.$$  
\end{Definition}

The first main result of this paper is the following.

\begin{thm}
The moduli spaces $\mathcal{A}_2(1,7)^{-}_{sym}$ and $\mathcal{A}_2(1,7)^{+}_{sym}$ are birational to the varieties $\VSP_6(F_4,6)$ and $\VSP^6(F_4,6)$ respectively, where $F_4\in k[x_0,x_1,x_2]_4$ is the Klein quartic. Furthermore, the moduli space $\mathcal{A}_2(1,7;2,2)$ is birational to $\VSP_{ord}(F_4,6)$.
\end{thm}

The apolarity theory developed in \cite{DK93} allows us to produce in Section \ref{VSPsection} a $3$-fold conic bundle dominating $\VSP_6(F_4,6)$, and to conclude that it is unirational. In Section \ref{birsection} we develop some birational geometry of the moduli spaces of abelian surfaces that we are considering. In particular, as a consequence of the above result, we get the following.

\begin{thm}
The moduli space $\mathcal{A}_2(1,7)^{-}_{sym}$ is unirational, and hence its Kodaira dimension is $-\infty$.
\end{thm}

At the end of the paper we also propose some open questions on the birational geometry of $\mathcal{A}_2(1,7)^{+}_{sym}$ and $\mathcal{A}_2(1,7;2,2)$. Throughout all the paper we will work over the complex field.

\subsection*{Acknowledgments}
The authors are members of the Gruppo Nazionale per le Strutture Algebriche, Geometriche e le loro Applicazioni of the Istituto Nazionale di Alta Matematica "F. Severi" (GNSAGA-INDAM). The first named author is member of the GDR "G\'eom\'etrie Alg\'ebrique G\'eom\'etrie Complexe" of the CNRS.

\section{Ordered varieties of sums of powers}\label{VSPsection}
Let $\nu_{d}^{n}:\mathbb{P}^{n}\rightarrow\mathbb{P}^{N(n,d)}$, with $N(n,d) =\binom{n+d}{d}-1$ be the Veronese embedding induced by $\mathcal{O}_{\mathbb{P}^{n}}(d)$, and let $V_{d}^{n} = \nu_{d}^{n}(\mathbb{P}^{n})$ be the corresponding Veronese variety. 

\begin{Definition}\label{vspord}
Let $F\in k[x_0,...,x_n]_{d}$ be a general homogeneous polynomial of degree $d$. Let $h$ be a positive integer and $\Hilb_h(\mathbb{P}^{n*})$ the Hilbert scheme of sets of $h$ points in $\mathbb{P}^{n*}$. We define 
$$\VSP(F,h)^{o} := \{\{L_{1},...,L_{h}\}\in\Hilb_{h}(\mathbb{P}^{n*})\: | \: F\in \langle L_{1}^d,...,L_{h}^d\rangle\subseteq\mathbb{P}^{N(n,d)}\}\subseteq \Hilb_{h}(\mathbb{P}^{n*})\},$$
and $\VSP(F,h) := \overline{\VSP(F,h)^{o}}$ by taking the closure of $\VSP(F,h)^{o}$ in $\Hilb_{h}(\mathbb{P}^{n*})$. 
\end{Definition}

Assume that the general polynomial $F\in\mathbb{P}^{N(n,d)}$ is contained in a $(h-1)$-linear space $h$-secant to $V_{d}^{n}$. Then, by \cite[Proposition 3.2]{Do04} the variety $\VSP(F,h)$ has dimension 
$$\dim(\VSP(F,h)) = h(n+1)-N(n,d)-1.$$
Furthermore, if $n = 1,2$ then for $F$ varying in an open Zariski subset of $\mathbb{P}^{N(n,d)}$ the variety $\VSP(F,h)$ is smooth and irreducible.

In order to apply these objects to the study of abelian surfaces, we need to construct similar varieties parametrizing decomposition of homogeneous polynomials as sums of powers of linear forms and admitting natural generically finite rational maps onto $\VSP(F,h)$.

\begin{Definition}\label{vsporc}
Let $F\in\mathbb{P}^{N(n,d)}$ be a general point. We define 
$$\VSP_{ord}(F,h)^{o} := \{(L_{1},...,L_{h})\in (\mathbb{P}^{n*})^{h} \: | \: 
F\in \langle L_{1}^d,...,L_{h}^d\rangle\subseteq\mathbb{P}^{N(n,d)}\}\subseteq (\mathbb{P}^{n*})^{h},$$
and $\VSP_{ord}(F,h) := \overline{\VSP_{ord}(F,h)^{o}}$ by taking the closure of $\VSP_{ord}(F,h)^{o}$ in $(\mathbb{P}^{n*})^{h}$. 
\end{Definition}

Note that $\VSP_{ord}(F,h)$ is a variety of dimension $h(n+1)-N(n,d)-1$. Furthermore, two general points of $\VSP_{ord}(F,h)$ define the same point of $\VSP(F,h)$ if and only if they differ by a permutation in the symmetric group $S_h$. Therefore, we have a generically finite rational map 
$$\phi_h:\VSP_{ord}(F,h)\dasharrow \VSP(F,h)$$ 
of degree $h!$ 

\begin{Remark}\label{smooth}
Arguing as in the proof of \cite[Proposition 3.2]{Do04}, with $(\mathbb{P}^{n*})^{h}$ instead of the Hilbert scheme $\Hilb_{h}(\mathbb{P}^{n*})$, we can show that if $n=1,2$ for a general polynomial $F$ the variety $\VSP_{ord}(F,h)$ is smooth and irreducible of dimension $h(n+1)-N(n,d)-1$. 
\end{Remark}

\begin{Definition}\label{vsph}
Consider the rational action of $S_{h-1}$ on $\VSP_{ord}(F,h)$ given by permuting the linear forms $(L_2,\dots,L_h)$. The variety $\VSP_{h}(F,h)$ is the quotient 
$$\VSP_{h}(F,h) = \VSP_{ord}(F,h)/S_{h-1}.$$
If $h = 2r$ is even consider the rational action of $S_r\times S_r$ on $\VSP_{ord}(F,h)$ where the first and the second copy of $S_r$ act on the first and the last $r$ linear forms respectively. Let $X_{h}(F) = \VSP_{ord}(F,h)/S_{r}\times S_r$. Therefore 
$$X_h(F)=\{(\{L_1,\dots, L_r\},\{L_{r+1},\dots,L_{h}\})\: |\: (L_1,\dots,L_h)\in \VSP_{ord}(F,h)\}$$ 
comes with a natural $S_2$-action switching $\{L_1,\dots, L_r\}$ and $\{L_{r+1},\dots,L_{h}\}$. We define 
$$\VSP^{h}(F,h) = X_h(F)/S_{2}.$$  
\end{Definition}
Note that $\VSP_{h}(F,h)$ admits a generically finite rational map 
$$\chi_h:\VSP_{h}(F,h)\dasharrow \VSP(F,h)$$ 
of degree $h$, and that the $h$ points in the fiber of $\chi_h$ over a general point $\{L_1,...,L_h\}\in \VSP(F,h)$ can be identified with the linear forms $L_1,...,L_h$ themselves. Similarly $\VSP^{h}(F,h)$ has a generically finite rational map 
$$\chi^h:\VSP^{h}(F,h)\dasharrow \VSP(F,h)$$ 
of degree $\frac{h!}{2(r!)^2}$.

The variety $\VSP_{h}(F,h)$ can be explicitly constructed in the following way. Let us consider the incidence variety 
\stepcounter{thm}
\begin{equation}\label{cost2}
\mathcal{J}:=\{(l,\{L_1,...,L_h\})\: | \: l\in\{L_1,...,L_h\}\in \VSP(F,h)^{o}\}\subseteq \mathbb{P}^{n*}\times \VSP(F,h)^{o}.
\end{equation}
Then $\VSP_{h}(F,h)$ is the closure $\overline{\mathcal{J}}$ of $\mathcal{J}$ in $\mathbb{P}^{n*}\times \VSP(F,h)$.

\begin{Remark}\label{maps}
In \cite{Mu92} Mukai proved that if $F\in k[x_0,x_1,x_2]_4$ is a general polynomial then $\VSP(F,6)$ is a smooth Fano $3$-fold $V_{22}$ of index $1$ and genus $12$. In this case we have a generically $720$ to $1$ rational map $\phi_6:\VSP_{ord}(F,6)\dasharrow \VSP(F,6)$, a generically $6$ to $1$ rational map $\chi_6:\VSP_6(F,6)\dasharrow \VSP(F,6)$, and a generically $10$ to $1$ rational map $\chi^6:\VSP^6(F,6)\dasharrow \VSP(F,6)$.

By \cite[Corollary 5.6]{GP01}, under the same assumptions on $F$, the moduli space $\mathcal{A}_2(1,7)$ of $(1,7)$-polarized abelian surfaces with canonical level structure is birational to $\VSP(F,6)$. As already observed in \cite[Theorem 4.4]{MS01} the Klein quartic 
\stepcounter{thm}
\begin{equation}\label{KQ}
F_4 = x_0^3x_1+x_1^3x_2+x_0x_2^3
\end{equation}
is general in the sense of Mukai \cite{Mu92}, hence the variety $\VSP(F_4,6)$ is isomorphic to the variety of sums of powers obtained for any other general quartic curve.
\end{Remark}

Let $V$ be a complex vector space of dimension $n+1$, choose coordinates $x_0,\dots,x_n$ in $V$ and the dual coordinates $\xi_0,\dots,\xi_n$ in $V^{*}$. Let $F\in k[x_0,\dots,x_n]_{d}$ be a homogeneous polynomial of even degree $d = 2m$, and consider the basis of $k[x_0,\dots,x_n]_{m}$ given by
\stepcounter{thm}
\begin{equation}\label{basis}
\mathcal{B}= \left\lbrace\binom{m}{m_0,\dots ,m_n}\prod_{t=0}^nx_t^{m_t},\: m_0+\dots + m_n = m\right\rbrace
\end{equation} 
where $\binom{m}{m_0,\dots ,m_n} = \frac{m!}{m_0!\dots m_n!}$. The $m$-th catalecticant matrix $Cat_m(F)$ of $F$ is the $(m+1)\times (m+1)$ symmetric matrix whose lines are the order $m$ partial derivatives of $F$ written in the basis $\mathcal{B}$ (\ref{basis}) in lexicographic order. The matrix $Cat_m(F)$ induces a symmetric bilinear form 
$$\Omega_F:k[\xi_0,\dots,\xi_n]_m\times k[\xi_0,\dots,\xi_n]_m\rightarrow k.$$ 
For our purposes the following result will be fundamental.

\begin{Lemma}\cite[Proposition 3.8]{Do04}\label{key}
Let $F\in k[x_0,\dots,x_n]_{d}$ be a homogeneous polynomial of even degree $d = 2m$ and assume that $F$ can be decomposed as 
$$F = L_1^{2m}+\dots +L_{h}^{2m}$$
and the powers $L_i^m$ are linearly independent in $k[x_0,\dots,x_n]_m$. Then $\Omega_F(L_i^m,L_j^m)=0$ for any $i,j = 1,\dots h$. 
\end{Lemma}

We refer to \cite[Section 2.3]{Do04} for further details on the bilinear form $\Omega_F$. 

\begin{Example}
The variety of sums of powers $\VSP(F_4,6)$ of the Klein quartic will play a central role in the rest of the paper. The second catalecticant matrix of $F_4$ is given by 
$$
Cat_2(F_4) = \left(\begin{array}{cccccc}
0 & 3 & 0 & 0 & 0 & 0 \\ 
3 & 0 & 0 & 0 & 0 & 0 \\ 
0 & 0 & 0 & 0 & 0 & 3 \\ 
0 & 0 & 0 & 0 & 3 & 0 \\ 
0 & 0 & 0 & 3 & 0 & 0 \\ 
0 & 0 & 3 & 0 & 0 & 0
\end{array}\right).
$$
and hence if $F_4 = L_1^4+\dots +L_6^4$, where $L_i = a_ix_0+b_ix_1+c_ix_2$ by Lemma \ref{key} we have that
\begin{equation}\label{eqKlein}
\Omega_{F_4}(L_i^2,L_j^2) = a_jb_ja_i^2+a_j^2a_ib_i+c_j^2a_ic_i+b_jc_jb_i^2+b_j^2b_ic_i+a_jc_jc_i^2=0. 
\end{equation}
\end{Example}

\begin{Definition}\label{defant}
For any $s=1,\dots, 6$ we define the variety $\mathcal{P}_s$ as the subvariety of $(\mathbb{P}^{2*})^{s}$ cut out by the conditions $\Omega_{F_4}(L_i^2,L_j^2) = 0$ for any $i,j=1,\dots,s$ with $i\neq j$. 
\end{Definition}

\begin{Proposition}\label{p1}
We have that $\mathcal{P}_2\subset (\mathbb{P}^{2*})^2$ is a smooth unirational $3$-fold conic bundle defined by a polynomial of bidegree $(2,2)$, and $\mathcal{P}_3\subset (\mathbb{P}^{2*})^3$ is an irreducible complete intersection $3$-fold defined by three polynomials of multidegree $(2,2,0),(2,0,2)$ and $(0,2,2)$. 
\end{Proposition}
\begin{proof}
Let $[a_i:b_i:c_i]$ be the homogeneous coordinates in the $i$-th factor $\mathbb{P}^2_i$ of $(\mathbb{P}^{2*})^s$. Therefore, a point in $\mathbb{P}^2_i$ with homogeneous coordinates $[a_i:b_i:c_i]$ corresponds to the linear form $L_i=a_ix_0+b_ix_1+c_ix_2$. 

By (\ref{eqKlein}) we have that if two general linear forms $L_i,L_j$ appear in a decomposition of $F$ then they must satisfy the following equation
$$
D_{i,j}=a_jb_ja_i^2+a_j^2a_ib_i+c_j^2a_ic_i+b_jc_jb_i^2+b_j^2b_ic_i+a_jc_jc_i^2=0.
$$
Note that the hypersurface $D_{i,j}=0$ is a divisor in $(\mathbb{P}^{2*})^s$ of multidegree $(d_1,\dots, d_s)$ with $d_i = d_j = 2$ and $d_r =0$ for any $r\neq i,j$. Therefore, $\mathcal{P}_2\subset (\mathbb{P}^{2*})^{2}$ is cut out by the equation 
\begin{equation}\label{eqA2}
D_{1,2}=a_1b_1a_2^2+a_1^2a_2b_2+c_1^2a_2c_2+b_1c_1b_2^2+b_1^2b_2c_2+a_1c_1c_2^2=0.
\end{equation}
Considering the nine standard affine charts covering $D_{1,2}$ we can prove that $D_{1,2}$ is smooth.

Note that any of the two projections onto the factors induces on $\mathcal{P}_2$ a structure of conic bundle over $\mathbb{P}^2$. If we choose for instance the projection on the first factor then (\ref{eqA2}) yields that the discriminant of the conic bundle is the smooth sextic given by 
$$C=\{a_1b_1^5+a_1^5c_1-5a_1^2b_1^2c_1^2+b_1c_1^5=0\}\subset\mathbb{P}^{2*}.$$ 
Now \cite[Corollary 1.2]{Me14} implies that $\mathcal{P}_2$ is unirational.

Let us consider the case $s = 3$. We have that $\mathcal{P}_3\subset (\mathbb{P}^{2*})^{3}$ is the complete intersection $3$-fold defined by the equations
\stepcounter{thm}
\begin{equation}\label{eqcomint}
\left\lbrace
\begin{array}{l}
D_{1,2} = a_1b_1a_2^2+a_1^2a_2b_2+c_1^2a_2c_2+b_1c_1b_2^2+b_1^2b_2c_2+a_1c_1c_2^2=0,\\ 
D_{1,3} = a_1b_1a_3^2+a_1^2a_3b_3+c_1^2a_3c_3+b_1c_1b_3^2+b_1^2b_3c_3+a_1c_1c_3^2=0, \\ 
D_{2,3} = a_2b_2a_3^2+a_2^2a_3b_3+c_2^2a_3c_3+b_2c_2b_3^2+b_2^2b_3c_3+a_2c_2c_3^2=0.
\end{array} 
\right.
\end{equation}
Finally, by a standard Macaulay2 \cite{Mc2} script we can show that $\mathcal{P}_3$ is irreducible.
\end{proof}

\begin{Proposition}\label{p2}
For $s=2,3$ there exist generically finite dominant morphisms 
$$f_s:\VSP_{ord}(F_4,6)\rightarrow \mathcal{P}_s$$ 
of degree $24$ and $6$ respectively. Furthermore, there is a morphism
$$g_3:\mathcal{P}_3\rightarrow\mathbb{P}^2$$ 
whose general fiber is a smooth curve of general type. Finally, there exists a generically finite rational map 
$$\alpha_{F_4}:\mathcal{P}_2\dasharrow\VSP_{6}(F_4,6)$$ 
of degree $5$. In particular $\VSP_{6}(F_4,6)$ is unirational. 
\end{Proposition}
\begin{proof}
By Lemma \ref{key} the image of the restriction of the projection $f_s:=\pi_{|\VSP_{ord}(F_4,6)}:\VSP_{ord}(F_4,6)\rightarrow (\mathbb{P}^{2*})^{s}$ is contained in $\mathcal{P}_s$. Let $L,M$ be two general linear forms appearing in a decomposition of $F_4$, and assume that 
$$F_4 = L^4+M^4+L_1^4+\dots +L_4^4 = aL^4+bM^4+M_1+\dots + M_4.$$
Set $G_1 = F_4-L^4-M^4$, $G_2 = F_4-aL^4-bM^4$, and consider the subspaces $H^{\partial}_{G_i}\subseteq\mathbb{P}(k[x_0,x_1,x_2]_2)$ generated by the second partial derivatives of $G_i$ for $i = 1,2$. Note that $H^{\partial}_{G_1}\subseteq \left\langle L_1^2,\dots, L_4^2\right\rangle$, $H^{\partial}_{G_2}\subseteq \left\langle M_1^2,\dots, M_4^2\right\rangle$. Furthermore, since $L^2,M^2$ appear in the six derivatives of $G_1$ and also in the six derivatives of $G_2$ we get that $H^{\partial}_{G_1}$ and $H^{\partial}_{G_2}$ must intersect in a $3$-plane. Hence
$$H^{\partial}_{G_1} = \left\langle L_1^2,\dots, L_4^2\right\rangle = H^{\partial}_{G_2} = \left\langle M_1^2,\dots, M_4^2\right\rangle.$$
Consider two general elements $C_1,C_2$ in the pencil of conics determined by the hyperplanes in $\mathbb{P}(k[x_0,x_1,x_2]_2)$ containing $H^{\partial}_{G_1}$. Then
$$L_i,M_i\in C_1\cap C_2$$
for $i = 1,\dots, 4$. Since the $L_i$'s are general $C_1$ and $C_2$ do not have a common irreducible component and hence $\{L_1,\dots,L_4\} = \{M_1,\dots,M_4\}$.

This means that a general decomposition of $F_4$ can be reconstructed just by knowing two of the linear forms appearing in it. In particular $f_2$, and a fortiori $f_3$, are generically finite onto their images. On the other hand, by Proposition \ref{p1} we have that $\mathcal{P}_2$ and $\mathcal{P}_3$ are irreducible $3$-folds, and this yields $\im(f_s) = \mathcal{P}_s$ when $s=2,3$. 

Furthermore, since a point in a general fiber of $f_s$ is determined up to a permutation of $6-s$ points we conclude that $f_2:\VSP_{ord}(F_4,6)\rightarrow \mathcal{P}_2$ is a generically finite dominant morphism of degree $24$, and $f_3:\VSP_{ord}(F_4,6)\rightarrow \mathcal{P}_3$ is a generically finite dominant morphism of degree $6$.

Let $g_3:\mathcal{P}_3\rightarrow\mathbb{P}^2$ be the restriction to $\mathcal{P}_3\subset(\mathbb{P}^{2*})^3$ of any of the projections, say the first one. A standard Macaulay2 \cite{Mc2} computation shows that the fiber of $g_3$ over $[a_1:b_1:c_1]=[1:1:1]$ is a smooth connected curve in $(\mathbb{P}^{2*})^2$. Therefore a general fiber $\Gamma$ of $g_3$ is a smooth connected curve as well. Note that by (\ref{eqcomint}) $\Gamma\subset (\mathbb{P}^{2*})^2$ is a smooth complete intersection defined by three polynomials of bi-degree $(2,0),(0,2),(2,2)$ respectively. By adjunction we get that the canonical sheaf of $\Gamma$ is given by 
$$\omega_{\Gamma} = \mathcal{O}_{\Gamma}(-3+2+0+2,-3+0+2+2) = \mathcal{O}_{\Gamma}(1,1)$$
and hence $\omega_{\Gamma}$ is ample.

Now, consider the case $s = 2$. Let $(L,M)\in\mathcal{P}_2$ be a general point. By the first part of the proof we know that $L,M$ determine a unique set of four linear forms $\{L_1,\dots,L_4\}$ such that $\{L,M,L_1,\dots,L_4\}$ gives a decomposition of $F_4$ as sum of six powers of linear forms. Therefore, keeping in mind (\ref{cost2}) we may define a dominant rational map
$$
\begin{array}{cccc}
\alpha_{F_4} & \mathcal{P}_2 & \dasharrow & \VSP_{6}(F_4,6)\\
 & (L,M) & \longmapsto & (L,\{L,M,L_1,\dots,L_4\})
\end{array}
$$   
of degree $5$. Finally, since by Proposition \ref{p1} $\mathcal{P}_2$ is unirational we conclude that $\VSP_{6}(F_4,6)$ is unirational as well.
\end{proof}

\section{Moduli of polarized abelian surfaces with level structure}

In this section we recall a couple of very specific results from \cite{BM16} that are needed here, and construct two new moduli spaces of abelian surfaces by introducing new arithmetic subgroups of $Sp_4(\mathbb{Z})$. For general results see \cite{BL04}, \cite{GP01}, \cite{Bo07} or the first four sections of \cite{BM16}. Since all the abelian surfaces we will deal with will be endowed with a polarization of type $(1,7)$, we will not mention any more this datum in the rest of the paper. Let $\mathbb{H}_2$ be the Siegel half space of abelian surfaces. 

\subsection{Arithmetic subgroups and quotients of $\mathbb{H}_2$}

Let $M_n(\mathbb{Z})$ be the space of $n\times n$ matrices with integer entries and let $D(1,7)\in M_{2}(\mathbb{Z})$ be the diagonal $2\times 2$ matrix 
$$
D(1,7)=\left(\begin{matrix}
1 & 0 \\ 
0 & 7
\end{matrix}\right).$$ 
We define the subgroup $\Gamma_{(1,7)}\subset M_{4}(\mathbb{Z})$ as:
\begin{equation}\label{gammad}
\Gamma_{(1,7)} := \left\lbrace \mathbf R\in M_{4}(\mathbb{Z}) \: | \: \mathbf R
\left(\begin{matrix}
0 & D(1,7) \\ 
-D(1,7) & 0
\end{matrix}\right), \mathbf R^t = 
\left(\begin{matrix}
0 & D(1,7) \\ 
-D(1,7) & 0
\end{matrix}\right)
\right\rbrace,
\end{equation}

and the subgroup $\Gamma_{(1,7)}(1,7)\subset\Gamma_{(1,7)}$ as

\begin{equation}\label{gammadd}
\Gamma_{(1,7)}(1,7):= \left\lbrace \left(\begin{matrix}
\mathbf{A} & \mathbf{B} \\ 
\mathbf{C} & \mathbf{D}
\end{matrix}\right)\in \Gamma_{(1,7)} \: | \: \mathbf{A}-I \equiv \mathbf{B} \equiv \mathbf{C} \equiv \mathbf{D}-I \equiv 0 \mod D(1,7)
\right\rbrace,
\end{equation}

where $\mathbf{M} \equiv 0 \mod D(1,7)$ if and only if $\mathbf{M}\in D(1,7)\cdot M_2(\mathbb{Z})$. See also \cite[Section 8.3.1]{BL04} for further details on this kind of  groups. 

Thanks to \cite[Section 8.2]{BL04} and the Baily-Borel theorem \cite{BB66}, since $\Gamma_{(1,7)}$ is a congruence arithmetic subgroup of the Siegel modular group, the quasi-projective variety $\mathcal{A}_{(1,7)} = \mathbb{H}_2/\Gamma_{(1,7)}$ is the moduli space of abelian surfaces endowed with a polarization of of type $(1,7)$ (the reader may see also \cite[Proposition 1.21]{HKW93}). Moreover, by \cite[Section 8.3]{BL04} and \cite{BB66}, the quasi-projective variety $\mathcal{A}_{(1,7)}(1,7) = \mathbb{H}_2/\Gamma_{(1,7)}(1,7)$ is the moduli space of abelian surfaces with a polarization of type $(1,7)$ and a canonical level $(1,7)$ structure.

On such an abelian surface, there exists $16$ symmetric line bundles representing the polarization. Let us denote by $A[2]$ the $\Z/2\Z$-vector space of $2$-torsion points of $A$, and let $H\in NS(A)$ be a polarization. We define a symmetric bilinear form $q^H: A[2]\times A[2] \to \{\pm 1\}$ as
$q^H(v,w):=\exp(\pi i E(2v,2w))$, where $E$ is the imaginary part of the Hermitian form of the polarization.

\begin{Definition}
A \textit{theta characteristic} is a quadratic form $q:A[2]\to \{\pm 1\}$ associated to $q^H$, i.e.
$$q(x)q(y)q(x+y)=q^H(x,y),$$
for all $x,y\in A[2]$.
\end{Definition}

In the following we will denote the set of theta characteristics by $\vartheta(A)$. To every symmetric line bundle $L$ we can associate a theta characteristic.
 
\begin{Definition}\label{bijtheta}
Let $L\in \Pic^H(A)$ be a symmetric line bundle, and $x \in A[2]$. We define $e^L(x)$ as the scalar $\beta$ such that
$\varphi(x):L(x) \stackrel{\sim}{\to} (\imath^*L)(x)=L(\imath(x))=L(x)$
is the multiplication by $\beta$. 
\end{Definition}

Alternatively, let $D$ be the symmetric divisor on $A$ such that $L\cong \mathcal{O}_A(D)$, the quadratic form $e^L$ can be defined as follows:
\begin{equation}\label{thetadiv}
e^L(x):=(-1)^{\mult_x(D)-\mult_0(D)}.
\end{equation}
From \cite[Lemma 4.6.2]{BL04} we observe that the set of theta characteristics on an abelian surface is a torsor under the action of $A[2]$, therefore it has cardinality $16$. In the case of a $(1,7)$-polarization there are $10$ even and $6$ odd line bundles \cite[Section 4.7]{BL04}. 

Recall that a theta structure induces in a natural way a canonical level structure, and that different theta structure may induce the same level structure. The $(1,7)$ case is peculiar, in this sense. In fact, we have the following lemma, which is a consequence of \cite[Lemmas 2.6 and 4.1]{BM16}.

\begin{Lemma}\label{levelsym}
Let $A$ be an abelian surface with a canonical $(1,7)$-level structure $\psi$. There exists a unique symmetric theta structure $\Psi$ that induces the level structure $\psi$. 
\end{Lemma}

Hence the datum of a level structure is equivalent to that of a symmetric theta structure. 
We will now study the action of the arithmetic subgroups previously defined on the set of symmetric line bundles, that admit a symmetric theta structure. In our particular case, by \cite[Section 6.9]{BL04} and \cite[Section 2]{BM16} each symmetric line bundle admits a unique symmetric theta structure. 
 
The set of symmetric theta divisors is in bijection with the set $(\frac{1}{2}\mathbb{Z}^2/\mathbb{Z}^2)$ of half-integer characteristics (\cite[Sections 4.6 and 4.7]{BL04} or \cite[Section 2]{Igu64}). The action of a symplectic matrix $M\in Sp_4(\mathbb{Z})$ on $\mathbb{H}_2$ induces an action on characteristics given by the following formula, for $a,b \in (\frac{1}{2}\mathbb{Z}^2/\mathbb{Z}^2)$:
\begin{equation}\label{char-action}
M\cdot 
\left(\begin{array}{c} a \\ b \end{array}\right) = \left( \begin{array}{cc} \textbf{D} & -\textbf{C} \\ -\textbf{B} & \textbf{A} \end{array}\right)\left(\begin{array}{c} a \\ b \end{array}\right)+\left(\begin{array}{c}
\diag(\textbf{CD}^t)\\ \diag(\textbf{AB}^t) \end{array}\right).
\end{equation}
\begin{Lemma}\cite[Section 2]{Igu64}
The action of $\Gamma_2$ on half-integer characteristics defined by formula (\ref{char-action}) has two orbits distinguished by the invariant
$$\mathbf{e}(m) = (-1)^{4ab^t}\in \{\pm 1\}.$$
\end{Lemma}
We will say that $m=(a,b) \in \frac{1}{2} \Z^4/\Z^4$ is an even (resp. odd) half-integer characteristic if $\mathbf{e}(m) = 1$ (resp. $\mathbf{e}(m) = -1$), and this notion of parity corresponds to those defined on symmetric theta divisors (or symmetric line bundles) and on theta-characteristics. Hence, in the following with a slight abuse of language we will call $e.g.$ an $odd$ (symmetric) line bundle the line bundle corresponding to an odd theta characteristic.

The modular group $Sp_4(\Z)$ acts on the set of theta characteristics through reduction modulo two, hence via $Sp_4(\Z/2\Z)$. Moreover, we have the following short exact sequence \cite[Lemma 4.2]{BM16}
$$1\mapsto \Gamma_{(1,7)}(2,14)\rightarrow \Gamma_{(1,7)}(1,7)\xrightarrow{r_2}Sp_4(\mathbb{Z}/2\mathbb{Z})\mapsto 1$$
where $r_2$ is in fact the reduction modulo two map. Let $O^{-}_4(\mathbb{Z}/2\mathbb{Z})\subset Sp_4(\mathbb{Z}/2\mathbb{Z})$ be the stabilizer of an odd quadratic form. There is an isomorphism $Sp_4(\mathbb{Z}/2\mathbb{Z})\cong S_6$, where $S_6$ is the symmetric group. Under this isomorphism $Sp_4(\mathbb{Z}/2\mathbb{Z})$ operates on the set of odd quadratic forms via permutations. Therefore, for the stabilizer subgroup of an odd theta characteristic we also have $O^{-}_4(\mathbb{Z}/2\mathbb{Z})\cong S_5\subset S_6$.  

\begin{Definition}
We denote by $\Gamma_2(1,7)^{-}$ the group 
$$\Gamma_2(1,7)^{-}:= r_2^{-1}(O^{-}_4(\mathbb{Z}/2\mathbb{Z}))\subset\Gamma_{(1,7)}(1,7)$$ that fits  in the exact sequence
$$1\to \Gamma_{(1,7)}(2,14)\rightarrow \Gamma_2(1,7)^{-}\xrightarrow{r_2}O^{-}_4(\mathbb{Z}/2\mathbb{Z})\to 1.$$
\end{Definition}
More explicitly
$$\Gamma_2(1,7)^{-}= \left\lbrace Z\in \Gamma_{(1,7)}(1,7) \: | \: Z\mod (2) \equiv \Sigma, \: \Sigma\in O^{-}_4(\mathbb{Z}/2\mathbb{Z})  \right\rbrace.$$
Hence, we have $\Gamma_{(1,7)}(2,14)\subset\Gamma_2(1,7)^{-}\subset\Gamma_{(1,7)}(1,7)$. Moreover, note that $|\Gamma_{(1,7)}(1,7):\Gamma_{(1,7)}(2,14)| = 6!$ and $|\Gamma_2(1,7)^{-}:\Gamma_{(1,7)}(2,14)| = 5!$ imply that
$|\Gamma_{(1,7)}(1,7):\Gamma_2(1,7)^{-}| = 6.$ Since $\Gamma_2(1,7)^{-}$ is a congruence arithmetic subgroup of $Sp_4(\Z)$, thanks to the Baily-Borel theorem \cite{BB66}, we obtain that the quotient 
$$\mathcal{A}_2(1,7)^{-}_{sym} := \mathbb{H}_2/\Gamma_2(1,7)^{-}$$
is a quasi-projective variety. The variety $\mathcal{A}_2(1,7)^{-}_{sym}$ is the moduli space of $(1,7)$-polarized abelian surfaces $(A,H)$ with an odd symmetric line bundle $L\in\Pic^H(A)$ and a $(1,7)$-level structure (or, which is the same, a symmetric theta structure for $L$). The degree of the morphism
$f^{-}:\mathcal{A}_2(1,7)^{-}_{sym}\rightarrow\mathcal{A}_{(1,7)}(1,7)$
that forgets the odd line bundle is $|\Gamma_{(1,7)}(1,7):\Gamma_2(1,7)^{-}| = 6$.

The same arguments hold, with slight modifications, if we want to construct moduli spaces for polarized abelian surfaces with level $(1,7)$ structure (or a symmetric theta structure) and an even line bundle in $\Pic^H(A)$.

Let $O^{+}_4(\mathbb{Z}/2\mathbb{Z})\subset Sp_4(\mathbb{Z}/2\mathbb{Z})$ be the stabilizer of an even quadratic form. 

\begin{Definition}
We denote by $\Gamma_2(1,7)^{+}$ the group 
$$\Gamma_2(1,7)^{+}:= r_2^{-1}(O^{+}_4(\mathbb{Z}/2\mathbb{Z}))\subset\Gamma_{(1,7)}(1,7)$$ 
that fits in the middle of the exact sequence
$$1\to \Gamma_{(1,7)}(2,14)\rightarrow \Gamma_2(1,7)^{+}\xrightarrow{r_2}O^{+}_4(\mathbb{Z}/2\mathbb{Z})\to 1.$$
\end{Definition}
The stabilizer subgroup $O^{+}_4(\mathbb{Z}/2\mathbb{Z})\subset Sp_4(\mathbb{Z}/2\mathbb{Z})$ of an even quadratic form is of order $|O^{+}_4(\mathbb{Z}/2\mathbb{Z})| = 72$ and $|\Gamma_{(1,7)}(1,7):\Gamma_2(1,7)^{+}| = 10.$
Using once again the Baily-Borel theorem \cite{BB66}, we have that the quotient 
$$\mathcal{A}_2(1,7)^{+}_{sym} := \mathbb{H}_2/\Gamma_2(1,7)^{+}$$
is a quasi-projective variety. It is the moduli space of $(1,7)$-polarized abelian surfaces $(A,H)$ with $(1,7)$-level structure and an even theta characteristic. Analogously to the odd case, the morphism $f^{+}:\mathcal{A}_2(1,7)^{+}_{sym}\rightarrow\mathcal{A}_{(1,7)}(1,7)$ forgetting the even theta characteristic has degree $|\Gamma_{(1,7)}(1,7):\Gamma_2(1,7)^{+}| = 10$.

These ideas are summarized in the following statement.

\begin{Proposition}
There exist arithmetic subgroups $\Gamma_2(1,7)^+$ and $\Gamma_2(1,7)^-$ such that there are quasi-projective moduli spaces 
\begin{center}
$\mathcal{A}_2(1,7)_{sym}^+:= \mathbb{H}_2/\Gamma_2(1,7)^+$ and $\mathcal{A}_2(1,7)_{sym}^-:= \mathbb{H}_2/\Gamma_2(1,7)^-$
\end{center}
that parametrize abelian surfaces with a $(1,7)$-structure, or equivalently a symmetric theta structure, and respectively an even or an odd theta characteristic.
\end{Proposition}

\begin{Definition}
We define
$$\Gamma_{(1,7)}(1,7;2,2):= \left\lbrace \left(\begin{matrix}
\mathbf{A} & \mathbf{B} \\ 
\mathbf{C} & \mathbf{D}
\end{matrix}\right)\in \Gamma_{(1,7)}(1,7)\; | \; \mathbf{A}-I \equiv \mathbf{B} \equiv \mathbf{C} \equiv \mathbf{D}-I\equiv 0\mod \left(\begin{matrix}
2 & 0 \\ 
0 & 2
\end{matrix}\right)
\right\rbrace$$
By \cite{BB66} and \cite[Section 8.3]{BL04}, the quasi-projective variety 
$$\mathcal{A}_2(1,7;2,2):= \mathbb{H}_2/ \Gamma_{(1,7)}(1,7;2,2)$$ 
is the moduli space of abelian surfaces with a polarization of type $(1,7)$, a level $(1,7)$-structure and a level $(2,2)$-structure. 
\end{Definition}

In the rest of the paper, while we will not change notation, all the moduli spaces considered will be non-singular models of suitable compactifications of the quasi-projective ones. 

\subsection{Theta-Null maps}
Let us now recall shortly how to construct theta-null maps for moduli spaces of $(1,7)$-polarized abelian surfaces, with a level structure and a theta-characteristic.
 
As we have explained in the preceding section, we identify theta-characteristics with symmetric line bundles, representing the polarization. Hence, we start from the datum $(A, H,L,\psi)$ of an abelian surface with a $(1,7)$-polarization $H$, a level structure $\psi$ and a symmetric line bundle $L\in \Pic(A)$ representing the polarization. As we have seen there exist $16$ symmetric line bundles, $10$ even and $6$ odd, inside the variety $\Pic^H(A)$ parametrizing line bundles that induce the polarization $H$. 

Let us pick the unique normalized linearization \cite[Section 2]{Mum66} on $L$ for the canonical involution $\pm Id$ on $A$, and call $H^0(A,L)^\pm$ the corresponding eigenspaces.
From Lemma \ref{levelsym} we know that the datum of a symmetric theta structure is equivalent to that of a level structure.  The upshot is the following:

\begin{itemize}
\item[-] if $L$ is even, the symmetric theta structure yields an identification $\Psi^+$ between $\P(H^0(A,L)^+)^*$ and an abstract $\P^3$. Similarly, we identify $\P(H^0(A,L)^-)^*$ with an abstract $\P^{2}$;
\item[-] if $L$ is odd, the symmetric theta structure yields an identification $\Psi^-$ between $\P(H^0(A,L)^+)^*$ and an abstract $\P^{2}$. In a similar way, we have an identification of $\P(H^0(A,L)^-)^*$ with an abstract $\P^{3}$.
\end{itemize}

The abstract projective spaces appearing in the preceding lists are all eigenspaces of involutions acting on the Scr\"odinger representation of the appropriate Heisenberg groups (see \cite[Section 5]{BM16} for details). 

Let $(A,L,\psi)$ be a polarized abelian surface $A$, with an even (respectively odd) line bundle $L$ representing the $(1,7)$-polarization, and a symmetric theta structure $\psi$. 

The Theta-Null maps (see \cite[Section 5.1]{BM16} for details and definitions) for abelian surfaces with a theta characteristic are defined as follows.
\begin{eqnarray*}
Th^+_{(1,7)}: \mathcal{A}_2(1,7)_{sym}^+ & \to & \mathbb{P}^3\\
(A,L,\psi) & \mapsto & \Psi^+(\Theta_{1,7}(0))
\end{eqnarray*}
\begin{eqnarray*}
Th^-_{(1,7)}: \mathcal{A}_2(1,7)_{sym}^- & \to & \mathbb{P}^2\\
(A,L,\psi) & \mapsto & \Psi^-(\Theta_{1,7}(0))
\end{eqnarray*}
where $\Theta_{1,7}(0)$ stays for the evaluation of global sections of $L$ at the origin. We remark that global sections that are anti-invariant with respect to the canonical involution on $A$ vanish at the origin. This is basically why we want to consider only invariant spaces of sections via the symmetric theta structure.

\section{Birational geometry of moduli spaces of (1,7)-polarized abelian surfaces}\label{birsection}
Recall that a proper variety $X$ over an algebraically closed field is rationally connected if two general points $x_1,x_2\in X$ can be joined by an irreducible rational curve. 

In \cite[Theorem 2]{BM16} we proved that the moduli space $\mathcal{A}_2(1,7)^{-}_{sym}$ of abelian surfaces with a symmetric theta structure and an odd theta characteristic is rationally connected. 

Clearly, rationality implies unirationality, which in turns implies rational connectedness. If $X$ is a smooth algebraic variety over an algebraically closed field of characteristic zero and $\dim(X)\leq 2$ these three notions are indeed equivalent \cite[Remark 1.3]{Ha01}. 

On the other hand, it is well known that a smooth cubic $3$-fold $X\subset\mathbb{P}^4$ is unirational but not rational \cite{CG72}. It is a long-standing open problem whether there exist varieties which are rationally connected but not unirational \cite[Section 1.24]{Ha01}. 

In this section, by using the techniques developed in Section \ref{VSPsection} we will prove that $\mathcal{A}_2(1,7)^{-}_{sym}$ is unirational.

\begin{thm}\label{Prop17}
Let $\mathcal{A}_2(1,7)^{-}_{sym}$ and $\mathcal{A}_2(1,7)^{+}_{sym}$ be the moduli spaces of abelian surfaces with level $(1,7)$-structure, a symmetric theta structure and an odd, respectively even theta characteristic. 

The moduli spaces $\mathcal{A}_2(1,7)^{-}_{sym}$ and $\mathcal{A}_2(1,7)^{+}_{sym}$ are birational to the varieties of sums of powers $\VSP_6(F_4,6)$ and $\VSP^6(F_4,6)$ respectively, where $F_4\in k[x_0,x_1,x_2]_4$ is the Klein quartic. Furthermore, the moduli space $\mathcal{A}_2(1,7;2,2)$ is birational to $\VSP_{ord}(F_4,6)$.
\end{thm}
\begin{proof}
Let $\mathcal{A}_2(1,7)$ be the moduli space of abelian surfaces with a $(1,7)$-level structure. Recall from \cite[Section 6.1.1]{BM16} that there exists a Theta-Null map $Th^-_{(1,7)}:\mathcal{A}_2(1,7)^{-}_{sym} \to \P^2$. 

By \cite{MS01} and \cite[Proposition 5.4 and Corollary 5.6]{GP01} there exists a birational map $\alpha: \mathcal{A}_2(1,7)\dasharrow \VSP(F_4,6)$ mapping a general $(A,\psi)\in\mathcal{A}_2(1,7)$ to the set $\{L_{1,A},\dots,L_{6,A}\}\in A$ of the odd $2$-torsion points of $A$, that are naturally mapped to $\P^2$ by the Theta-Null map. The six points $\{L_{1,A},\dots,L_{6,A}\}$ of $\P^2$ give a decomposition in $\VSP(F_4,6)$.

Each of the $6$ odd $2$-torsion points correspond to a choice of an odd theta characteristic via $Th^-_{(1,7)}$. Now, consider a general point $(A,\psi,L)$ of $\mathcal{A}_2(1,7)^{-}_{sym}$ over $(A,\psi)\in\mathcal{A}_2(1,7)$. We may define a rational map $\beta:\mathcal{A}_2(1,7)^{-}_{sym}\dasharrow\VSP_6(F_4,6)$ sending $(A,\psi,L)$ to the linear form in $\chi_6^{-1}(\{L_{1,A},...,L_{6,A}\})$ that corresponds to $Th^-_{(1,7)}(A,\psi,L)\in \P^2$, where $\chi_6$ is the map in Remark \ref{maps}. To conclude it is enough to observe that since $\alpha$ is birational the map $\beta:\mathcal{A}_2(1,7)^{-}_{sym}\dasharrow \VSP_6(F_4,6)$ is birational as well.

In particular, since the generic abelian surface is Jacobian, odd theta characteristics correspond to the Weierstrass points of the corresponding curve. It is a classical fact that even theta characteristics correspond to partitions of the Weierstrass points into two $3$-elements sets, see for example \cite[Chapter 8]{DO88}. This directly implies that
$\mathcal{A}_2(1,7)^{+}_{sym}$ is birational to $\VSP^6(F_4,6)$. 

Finally, recall that a $(2,2)$-level structure for a Jacobian abelian surface corresponds to a complete ordering of the Weierstrass points \cite[Chapter 8]{DO88}. This in turn implies that the moduli space $\mathcal{A}_2(1,7;2,2)$ is birational to $\VSP_{ord}(F_4,6)$.
\end{proof}
We summarize the situation in the following diagram, where the superscripts on the arrows indicate the degrees of the respective maps.
\[
  \begin{tikzpicture}[xscale=3.5,yscale=-1.5]
    \node (A0_0) at (0, 0) {};
    \node (A0_1) at (1, 0) {$\VSP_{ord}(F_4,6)\overset{\rm{bir}}{\cong}\mathcal{A}_2(1,7;2,2)$};
    \node (A1_0) at (0, 1) {$\mathcal{P}_3$};
    \node (A1_1) at (1, 1) {$\mathcal{P}_2$};
    \node (A1_2) at (2, 1) {$\VSP_{6}(F_4,6)\overset{\rm{bir}}{\cong}\mathcal{A}_2(1,7)_{sym}^{-}$};
    \node (A2_0) at (0, 2) {$\VSP^{6}(F_4,6)\overset{\rm{bir}}{\cong}\mathcal{A}_2(1,7)_{sym}^{+}$};
    \node (A2_2) at (2, 2) {$\VSP(F_4,6)\overset{\rm{bir}}{\cong}\mathcal{A}_2(1,7)$};
    \path (A0_1) edge [->,dashed]node [auto] {$\scriptstyle{5!}$} (A1_2);
    \path (A0_1) edge [->]node [auto,swap] {$\scriptstyle{3!}$} (A1_0);
    \path (A1_0) edge [->,dashed]node [auto] {$\scriptstyle{120}$} (A2_2);
    \path (A1_0) edge [->]node [auto] {$\scriptstyle{4}$} (A1_1);
    \path (A1_1) edge [->,dashed]node [auto] {$\scriptstyle{5}$} (A1_2);
    \path (A1_1) edge [->,dashed]node [auto] {$\scriptstyle{30}$} (A2_2);
    \path (A1_0) edge [->,dashed]node [auto,swap] {$\scriptstyle{12}$} (A2_0);
    \path (A2_0) edge [->,dashed]node [auto] {$\scriptstyle{10}$} (A2_2);
    \path (A0_1) edge [->,swap]node [auto] {$\scriptstyle{4!}$} (A1_1);
    \path (A1_2) edge [->,dashed]node [auto] {$\scriptstyle{6}$} (A2_2);
  \end{tikzpicture}
  \]
\begin{thm}\label{A17}
The moduli space $\mathcal{A}_2(1,7)^{-}_{sym}$ is unirational, and hence its Kodaira dimension is $-\infty$. 
\end{thm}
\begin{proof}
Since by Proposition \ref{p2} $\VSP_6(F_4,6)$ is unirational the claim follows from Proposition \ref{Prop17}.
\end{proof}

\begin{Remark}
The moduli space $\mathcal{A}_2(1,7;2,2)$ admits a rational fibration over $\mathbb{P}^2$ whose general fiber is a curve of general type. 

Indeed, the fibers of the restriction to $\VSP_{ord}(F_4,6)$ of the first projection $\pi_{|\VSP_{ord}(F_4,6)}:\VSP_{ord}(F_4,6)\rightarrow\mathbb{P}^2$ are mapped by $f_3:\VSP_{ord}(F_4,6)\rightarrow \mathcal{P}_3$ onto the fibers of the morphism $g_3:\mathcal{P}_3\rightarrow\mathbb{P}^2$. 

By Proposition \ref{p2} the general fiber of $g_3$ is a curve of general type. Hence the general fiber of $\pi_{|\VSP_{ord}(F_4,6)}$ is of general type as well, and by Proposition \ref{Prop17} $\pi_{|\VSP_{ord}(F_4,6)}$ induces a rational fibration of $\mathcal{A}_2(1,7;2,2)$ over $\mathbb{P}^2$ whose general fiber is a curve of general type.
\end{Remark}

\subsection{Questions}
We close the paper with some open questions on the birational type of the moduli spaces $\mathcal{A}_2(1,7;2,2)$ and $\mathcal{A}_2(1,7)_{sym}^{+}$ or equivalently of the varieties of sums of powers $\VSP_{ord}(F_4,6)$ and $\VSP^{6}(F_4,6)$.
\begin{Question}
\textit{Are $\mathcal{A}_2(1,7;2,2)$ and $\mathcal{A}_2(1,7)_{sym}^{+}$ varieties of general type?}
\end{Question}
We would like to mention that the analogous problem for the variety of sums of powers $\VSP_{ord}(F_3,4)$ where $F_3\in k[x_0,x_1,x_2]_3$ is a general cubic polynomial, has been studied in \cite{RR93}. Indeed, by \cite[Theorem 8.3]{RR93} we have that $\VSP_{ord}(F_3,4)$ is of general type.

Note that by Proposition \ref{p2} we have a finite morphism 
$$f_3:\VSP_{ord}(F_4,6)\rightarrow\mathcal{P}_3\subseteq (\mathbb{P}^{2*})^3$$
where $\mathcal{P}_3\subseteq (\mathbb{P}^{2*})^3$ is a complete intersection irreducible $3$-fold cut out by equations of multi-degree $(2,2,0)$, $(2,0,2)$ and $(0,2,2)$. Therefore, by adjunction the dualizing sheaf of $\mathcal{P}_3$ is
$$\omega_{\mathcal{P}_3}\cong\mathcal{O}_{\mathcal{P}_3}(1,1,1).$$
We checked using Macaulay2 \cite{Mc2} that $\mathcal{P}_3$ is singular along a curve.
\begin{Question}
\textit{Are the singularities of the $3$-fold $\mathcal{P}_3$ at worst canonical?} 
\end{Question}
Note that a positive answer to this last question would imply that $\mathcal{P}_3$, and hence $\VSP_{ord}(F_4,6)$, are of general type.  

\bibliographystyle{amsalpha}
\bibliography{Biblio}

\end{document}